\documentclass[11pt]{amsart}

\usepackage[T1]{fontenc}
\usepackage[utf8]{inputenc}
\usepackage{lmodern}
\usepackage{amsmath, amssymb, amsthm, mathtools}
\usepackage{enumitem}
\usepackage{microtype}
\usepackage{xcolor}
\usepackage[margin=1.15in]{geometry}
\usepackage{hyperref}

\hypersetup{
    colorlinks = true,
    linkcolor = blue!55!black,
    citecolor = blue!55!black,
    urlcolor = blue!55!black
}

\numberwithin{equation}{section}

\newtheorem{theorem}{Theorem}[section]
\newtheorem{proposition}[theorem]{Proposition}
\newtheorem{lemma}[theorem]{Lemma}
\newtheorem{corollary}[theorem]{Corollary}
\newtheorem*{theorem*}{Theorem}
\theoremstyle{definition}

\newtheorem{example}[theorem]{Example}
\newtheorem{remark}[theorem]{Remark}

\newcommand{\Z}{\mathbb Z}

\newcommand{\C}{\mathbb C}

\title{Strongly Primitive Salem Growth Polynomials for Right-Angled Coxeter Groups}
\author{Mingyu Oh}

\address{Department of Mathematics, Chung-Ang University, Seoul, Republic of Korea}

\email{ohmk1228@cau.ac.kr}

\date{October 30, 2026}

\subjclass[2020]{Primary 20F55; Secondary 05C31, 05C69, 11R06}

\keywords{right-angled Coxeter groups, spherical growth series, growth polynomials, Salem numbers, clique polynomials, Lehmer's polynomial}

\begin{document}

\begin{abstract}
    We study standard spherical growth rates of right-angled Coxeter groups through the clique polynomial of the defining graph. We prove that every even degree at least four occurs as the degree of a strongly primitive Salem growth rate: for each $d\geq2$, there are infinitely many connected $K_{2d+1}$-free defining graphs whose full reciprocal-radius polynomial is an irreducible Salem polynomial of degree $2d$. We also establish independence-polynomial obstructions for prescribed Salem polynomials, including a sharp first-coefficient bound $a_1\leq-5$. Finally, an elementary golden-ratio growth gap excludes Lehmer's number from the standard spherical growth rates of right-angled Coxeter groups.
\end{abstract}

\maketitle

\section{Introduction}
A \emph{Salem number} is a real algebraic integer $\tau>1$ all of whose
Galois conjugates lie on or inside the unit circle and at least one of whose
conjugates lies on the unit circle. Its minimal polynomial is monic, reciprocal, and of even degree at least four. Salem numbers occur naturally in questions around Lehmer's problem, Coxeter growth, hyperbolic geometry, dynamics, and knot theory; see the survey of Ghate and Hironaka \cite{GhateHironaka}. The smallest known Salem number is Lehmer's number
$$\lambda_0=1.1762808\ldots,$$
the largest real zero of Lehmer's polynomial
\begin{equation}\label{eq:Lehmer}
    \mathcal{L}(x)=x^{10}+x^9-x^7-x^6-x^5-x^4-x^3+x+1.
\end{equation}
Extensive computations have not found a polynomial of smaller Mahler measure; see Mossinghoff \cite{Mossinghoff} for one such search and for background on small Mahler measures.

Salem numbers enter Coxeter theory through several different mechanisms. Floyd and Parry studied appearances of Salem numbers in Coxeter growth, especially in geometric and hyperbolic reflection settings. Floyd related growth exponents of planar Coxeter reflection groups to Pisot and Salem numbers \cite{Floyd}, while Parry studied growth series of Coxeter groups through reciprocal and antireciprocal rational functions and Salem polynomial factors \cite{Parry}. McMullen studied a different Coxeter-theoretic source, namely spectral radii of elements in the geometric representation, and proved a Lehmer-type lower bound in that setting \cite{McMullen}. Kolpakov and Talambutsa proved that the standard spherical growth rate of an infinite right-angled Coxeter group is either $1$ or a Perron number \cite{KolpakovTalambutsa}. This general Perron property does not decide which Salem numbers can occur, or whether a prescribed Salem polynomial can be the full reciprocal-radius polynomial of a defining graph. The present paper addresses the latter, more restrictive realization problem through clique polynomials of defining graphs.

This paper studies Salem numbers arising from \emph{right-angled} Coxeter groups. If $\Gamma$ is a finite graph, the associated right-angled Coxeter group is
$$W_\Gamma=\left\langle s_v (v \in V(\Gamma))\mid s_v^2=1, s_vs_w=s_ws_v\text{ whenever } \{v,w\}\in E(\Gamma)\right\rangle.$$
We use the standard generating set $S=V(\Gamma)$. Let
$$W_\Gamma(t) = \sum_{n\geq0}a_nt^n$$
be the spherical growth series with respect to $S$. The standard growth rate is the reciprocal of the radius of convergence of $W_\Gamma(t)$. Steinberg's formula \cite[Corollary 1.29]{Steinberg} specializes in the right-angled case to a formula involving only the clique numbers of $\Gamma$; see also Davis \cite[Chap. 17]{Davis}, Glover--Scott \cite{GloverScott}, and Athreya--Prasad \cite{AthreyaPrasad}. If $c_k$ is the number of $k$-cliques in $\Gamma$, with $c_0=1$, and if $r$ is the clique number of $\Gamma$, then the full reciprocal-radius growth polynomial is
\begin{equation}
    P_\Gamma(x)=\sum_{k=0}^r(-1)^kc_k(x+1)^{r-k}.
\end{equation}
In the exponential case, the growth rate is the largest positive root of the non-cancelled reciprocal-radius denominator. In the constructions of this paper there is no cancellation, because the polynomial $P_\Gamma$ is irreducible.

We call a Salem growth rate \emph{strongly primitive} if its Salem minimal polynomial is equal to the full reciprocal-radius polynomial $P_\Gamma$, not merely to a proper factor of it. Our first main result says that strongly primitive Salem growth is abundant.
\begin{theorem}\label{thm:main1}
    For every integer $d\geq2$, there are infinitely many connected $K_{2d+1}$-free graphs $\Gamma$ such that $P_\Gamma$ is an irreducible Salem polynomial of degree $2d$. Consequently every even degree $2d\geq 4$ occurs as the degree of a strongly primitive standard growth rate of a right-angled Coxeter group.
\end{theorem}
We construct the defining graph by starting with $O_{2d}=K_{2,2,\ldots,2}$, the one-skeleton of the $2d$-dimensional cross-polytope, and then adding vertices whose links have prescribed clique polynomials. The resulting polynomial is
\begin{equation}
    P_{d,p}(x)=(x-1)^{2d}-px\sum_{\ell=0}^{d-1}b_{d,\ell}(x-1)^{2\ell}(x+1)^{2d-2-2\ell},
\end{equation}
where 
$$B_d(z)=\prod_{j=1}^{d-1}(z+j)=\sum_{\ell=0}^{d-1}b_{d,\ell}z^\ell.$$
For all sufficiently large primes $p>d-1$, the shift $P_{d,p}(y+1)$ is Eisenstein at $p$, and the roots have the Salem root location. The point of the construction is that the graph realizes a reciprocal unit polynomial by design, while the parameter $p$ separates the real roots after passing to the $x+x^{-1}$ variable.

We next turn to restrictions on which Salem polynomials can occur. For a monic polynomial $M$ of degree $D$, define
\begin{equation}
    A_M(u)=(-u)^D M\!\left(-1-\frac{1}{u}\right).
\end{equation}
If $M=P_\Gamma$, then $A_M$ is the clique polynomial of $\Gamma$, equivalently the independence polynomial of the complement graph $\overline\Gamma$. This converts the problem of realizing a Salem polynomial into a graph-theoretic problem about independence polynomials.

\begin{theorem}\label{thm:main2}
    Let
    $$M(x)=x^D+a_1x^{D-1}+\cdots+a_1x+1$$
    be a Salem minimal polynomial. If $M$ is the full reciprocal-radius polynomial $P_\Gamma$ for some finite graph $\Gamma$, then
    $$a_1\leq-5.$$
    Moreover,
    $$0<(-1)^DM(-1)\leq 2^{-a_1}.$$
    The bound on $a_1$ is sharp.
\end{theorem}
The inequality $a_1\leq -5$ is obtained by combining the inverse transform with a small vertex-cover analysis for independence polynomials. The sharpness is witnessed by the quartic Salem polynomial
$$x^4-5x^3+4x^2-5x+1,$$
which occurs as the full reciprocal-radius polynomial of a right-angled Coxeter group. We also apply the same obstruction directly to Lehmer's polynomial. It shows that $\mathcal{L}(x)$ itself is not strongly primitive.

\subsection*{Organization}
Section \ref{sec:growth} recalls the clique formula for right-angled Coxeter full reciprocal-radius growth polynomials and the relation with $h$-polynomials. Section \ref{sec:salem} gives the reciprocal $x+x^{-1}$ criterion for Salem polynomials. Section \ref{sec:construction} proves the all-even-degree construction. Section \ref{sec:obstruction} introduces the inverse transform and the elementary independence-polynomial obstruction. Section \ref{sec:firstcoefficient} proves the sharp first-coefficient obstruction $a_1\leq-5$. Section \ref{sec:Lehmer} establishes the golden-ratio lower bound for exponential standard growth rates and applies it to Lehmer's number.

\section{Growth polynomials from clique data}\label{sec:growth}

\subsection{The clique formula}

Let $\Gamma$ be a finite simple graph. A subset $T\subseteq V(\Gamma)$ generates a finite special subgroup of $W_\Gamma$ if and only if $T$ spans a clique in $\Gamma$. In that case $W_T\cong(\Z/2\Z)^{|T|}$, and its growth polynomial with respect to $T$ is $(1+t)^{|T|}$. The formula below is the right-angled specialization of Steinberg's growth formula \cite[Corollary~1.29]{Steinberg}; modern accounts and graph-theoretic formulations can be found in Davis \cite[Chap.~17]{Davis}, Glover--Scott \cite[Prop.~3]{GloverScott}, and Athreya--Prasad \cite[Eq.~(5)]{AthreyaPrasad}.

Let
$$c_k=c_k(\Gamma)=\#\{\text{cliques of size $k$ in $\Gamma$}\},\qquad c_0=1$$
and let $r=\omega(\Gamma)$ be the clique number of $\Gamma$, the largest $k$ with $c_k\neq0$. The clique polynomial is
$$C_\Gamma(u)=\sum_{k=0}^rc_ku^k.$$

\begin{proposition}\label{prop:SteinbergRA}
    The standard spherical growth series of $W_\Gamma$ satisfies
    \begin{equation}\label{eq:steinberg}
        \frac{1}{W_\Gamma(t^{-1})}=\sum_{k=0}^{r} c_k\frac{(-1)^k}{(1+t)^k}.
    \end{equation}
    Equivalently,
    \begin{equation}\label{eq:denominator}
        W_\Gamma(t)=\frac{(1+t)^r}{D_\Gamma(t)},\qquad D_\Gamma(t)=\sum_{k=0}^{r}(-1)^k c_k t^k(1+t)^{r-k},
    \end{equation}
    up to cancellation of common factors. The reciprocal-radius polynomial is
    \begin{equation}\label{eq:Pclique}
        P_\Gamma(x)=x^rD_\Gamma(1/x)=\sum_{k=0}^{r}(-1)^k c_k(x+1)^{r-k}.
    \end{equation}
\end{proposition}

\begin{proof}
Steinberg's formula for Coxeter systems \cite[Corollary~1.29]{Steinberg} states that
\[
   \frac{1}{W_\Gamma(t^{-1})}
   =\sum_{T\subseteq S:\ W_T\text{ finite}}\frac{(-1)^{|T|}}{W_T(t)}.
\]
In the right-angled case the finite special subgroups are exactly the clique subgroups. A clique of size $k$ contributes $(-1)^k/(1+t)^k$, and there are $c_k$ such cliques. This proves \eqref{eq:steinberg}. Multiplying by $(1+t)^r$ gives \eqref{eq:denominator}. Finally substituting $t=1/x$ and multiplying by $x^r$ gives \eqref{eq:Pclique}.
\end{proof}

\begin{remark}
    The phrase \emph{full reciprocal-radius growth polynomial} will mean the polynomial $P_\Gamma$ in \eqref{eq:Pclique}, before deleting any factor. In the main construction $P_\Gamma$ is irreducible, so no ambiguity arises.
\end{remark}

\subsection{\texorpdfstring{The $h$-polynomial form}{The h-polynomial form}}
Let $L=L(\Gamma)$ be the clique complex of $\Gamma$. Its $(k-1)$-faces are the $k$-cliques of $\Gamma$. We use the descending convention
\begin{equation}\label{eq:hpoly}
    H_L(y)=\sum_{k=0}^r c_k(y-1)^{r-k}.
\end{equation}
Equivalently, if $h_L(t)$ denotes the standard $h$-polynomial, then
$$H_L(y)=y^r h_L(1/y).$$
Then
\begin{equation}\label{eq:htransform}
    P_\Gamma(x)=(-1)^rH_L(-x).
\end{equation}
Indeed, substituting $y=-x$ into \eqref{eq:hpoly} gives
$$H_L(-x)=\sum_{k=0}^rc_k(-x-1)^{r-k}=(-1)^r\sum_{k=0}^r(-1)^kc_k(x+1)^{r-k}.$$
Thus the arithmetic of $P_\Gamma$ is the arithmetic of $H_L(-x)$.

\begin{remark}
    The appearance of reciprocal polynomials here is related to the classical reciprocity theory of Coxeter growth functions. Charney and Davis proved that if the proper nerve is an Euler sphere of the appropriate type, then the Coxeter growth function satisfies a reciprocity relation $W(t^{-1})=\pm W(t)$ \cite{CharneyDavis}. By (2.5), the palindromicity condition $H_L(y)=y^rH_L(y^{-1})$ is equivalent to
    \[
        x^rP_\Gamma(x^{-1})=(-1)^rP_\Gamma(x).
    \]
    Thus a palindromic $h$-vector corresponds to a reciprocal $P_\Gamma$ when $r$ is even and to an anti-reciprocal one when $r$ is odd. Indeed, $P_\Gamma(-1)=(-1)^rc_r\neq0$, so $P_\Gamma$ cannot be reciprocal when $r$ is odd. Our construction has $r=2d$ and does not require the nerve to be a sphere or a manifold; reciprocity is realized directly by the graph construction.
\end{remark}

\subsection{Joins}
If $\Gamma$ and $\Lambda$ are graphs, their join $\Gamma*\Lambda$ is obtained by adding all edges between $V(\Gamma)$ and $V(\Lambda)$. Cliques in a join are unions of cliques, hence
\begin{equation}\label{eq:cliquejoin}
    C_{\Gamma*\Lambda}(u)=C_\Gamma(u)C_\Lambda(u).
\end{equation}
Consequently,
\begin{equation}\label{eq:Pjoin}
    P_{\Gamma*\Lambda}(x)=P_\Gamma(x)P_\Lambda(x),
\end{equation}
where $P_\Gamma$, $P_\Lambda$, and $P_{\Gamma*\Lambda}$ are defined using the effective degrees $\omega(\Gamma)$, $\omega(\Lambda)$, and $\omega(\Gamma)+\omega(\Lambda)$, respectively.

In particular, let $O_m=K_{2,2,\ldots,2}$ be the complete multipartite graph with $m$ parts, each of size two. We also set $O_0$ to be the empty graph, with clique polynomial 1. Then
$$C_{O_m}(u)=(1+2u)^m, \qquad P_{O_m}(x)=(x-1)^m.$$
Joining $O_m$ therefore multiplies the full reciprocal-radius growth polynomial by $(x-1)^m$.
\section{Reciprocal polynomials and the Salem criterion}\label{sec:salem}
A polynomial $P(x)$ of degree $2d$ is \emph{reciprocal} if $P(x)=x^{2d}P(1/x)$. If $P$ is monic, reciprocal, and has constant term 1, then there is a unique monic polynomial $R(y)\in\Z[y]$ of degree $d$ such that
\begin{equation}\label{eq:Rdef}
    x^{-d}P(x)=R(x+x^{-1}).
\end{equation}
This follows from the standard identities
$$x^j+x^{-j}\in\Z[x+x^{-1}].$$
\begin{proposition}\label{prop:SalemCriterion}
    Let $P(x)\in\Z[x]$ be monic, irreducible, reciprocal, and of degree $2d\geq4$. Let $R(y)$ be defined by \eqref{eq:Rdef}. Then $P$ is a Salem polynomial if and only if $R$ has exactly one root in $(2,\infty)$ and exactly $d-1$ roots in $(-2,2)$, counted with multiplicity. 
\end{proposition}
\begin{proof}
    If $x\in\C^\times$ and $y=x+x^{-1}$, then $|x|=1$ with $x\ne\pm1$ is equivalent to $y\in(-2,2)$. If $x>1$, then $y=x+x^{-1}>2$, and the two reciprocal roots $x$ and $x^{-1}$ determine the same $y$. Thus a root $y_0>2$ of $R$ gives two positive reciprocal roots of $P$, one larger than $1$ and one smaller than $1$. A root $y_0\in(-2,2)$ gives a conjugate pair on the unit circle. Since $P$ is irreducible of degree at least 4, it has no root $\pm1$. Thus the stated locations of the roots of $R$ are equivalent to the Salem root location of $P$.
\end{proof}

\begin{corollary}\label{cor:oddgap}
    An irreducible full reciprocal-radius growth polynomial of odd degree cannot be a Salem polynomial. Hence new strongly primitive Salem growth can occur only in even effective degree.
\end{corollary}
\begin{proof}
    The minimal polynomial of a Salem number is reciprocal and has even degree. An irreducible polynomial of odd degree cannot be such a polynomial.
\end{proof}
\section{Proof of Theorem \ref{thm:main1}}\label{sec:construction}
\subsection{The graph construction}
Fix $d\geq 2$. Define
\begin{equation}\label{eq:Bd}
    B_d(z)=\prod_{j=1}^{d-1}(z+j)=\sum_{\ell=0}^{d-1}b_{d,\ell}z^\ell.
\end{equation}
All coefficients $b_{d,\ell}$ are positive integers, and $b_{d,0}=(d-1)!$.

Let $p$ be a positive integer. Start with $O_{2d}$. For each $\ell=0,\ldots,d-1$, add $pb_{d,\ell}$ new vertices. These new vertices are pairwise non-adjacent. Each new vertex corresponding to $\ell$ is joined to exactly the vertices of an induced subgraph of $O_{2d}$ isomorphic to $K_1*O_{2\ell}$.

For $\ell=0$, this means that the new vertex is attached to a single vertex of $O_{2d}$. Such induced subgraphs exist because $O_{2d}$ has $2d$ parts and $1+2\ell\leq 2d-1$.

Let $\Gamma_{d,p}$ be the resulting graph.

\begin{lemma}\label{lem:realization}
    The graph $\Gamma_{d,p}$ is connected and $K_{2d+1}$-free. Its full reciprocal-radius growth polynomial is
    \begin{equation}\label{eq:Pdp}
        P_{\Gamma_{d,p}}(x)=(x-1)^{2d}-px\sum_{\ell=0}^{d-1}b_{d,\ell}(x-1)^{2\ell}(x+1)^{2d-2-2\ell}.
    \end{equation}
\end{lemma}
\begin{proof}
    The graph $O_{2d}$ is connected. Every added vertex is adjacent to at least one vertex of $O_{2d}$, so $\Gamma_{d,p}$ is connected.

    The largest clique in $O_{2d}$ has size $2d$. A clique containing an added vertex corresponding to $\ell$ is obtained by adjoining that vertex to a clique in its link $K_1*O_{2\ell}$, whose clique number is $1+2\ell$. Hence such a clique has size at most $2+2\ell\leq 2d$. Since added vertices are pairwise non-adjacent, no clique contains two of them. Thus $\Gamma_{d,p}$ is $K_{2d+1}$-free.

    It remains to compute the polynomial. The base graph has
    $$C_{O_{2d}}(u)=(1+2u)^{2d}, \qquad P_{O_{2d}}(x)=(x-1)^{2d}.$$
    Adding a vertex whose link has clique polynomial $C_H(u)$ adds $uC_H(u)$ to the clique polynomial. Here
    $$C_{K_1*O_{2\ell}}(u)=(1+u)(1+2u)^{2\ell}.$$
    In degree $2d$, the contribution of one such vertex to $P$ is
    $$(x+1)^{2d}\left(-\frac{1}{x+1}\right)\left(1-\frac{1}{x+1}\right)\left(1-\frac{2}{x+1}\right)^{2\ell}=-x(x-1)^{2\ell}(x+1)^{2d-2-2\ell}.$$
    Multiplying by $pb_{d,\ell}$ and summing over $\ell$ gives \eqref{eq:Pdp}.
\end{proof}
For brevity we write $P_{d,p}(x):=P_{\Gamma_{d,p}}(x)$.

\begin{lemma}\label{lem:eisenstein}
    Assume that $p$ is an odd prime with $p>d-1$. Then $P_{d,p}$ is monic, reciprocal, has constant term 1, and is irreducible over $\Z$.
\end{lemma}
\begin{proof}
    The polynomial $(x-1)^{2d}$ is reciprocal of degree $2d$. Each summand
    $$x(x-1)^{2\ell}(x+1)^{2d-2-2\ell}$$
    is also reciprocal when regarded as a polynomial of degree $2d$, because
    $$x^{2d}\left( x^{-1}(x^{-1}-1)^{2\ell}(x^{-1}+1)^{2d-2-2\ell}\right)=x(x-1)^{2\ell}(x+1)^{2d-2-2\ell}.$$
    Thus $P_{d,p}$ is reciprocal. It is monic because the subtracted summands have degree at most $2d-1$, and it has constant term 1 because the subtracted summands are divisible by $x$.

    Set $x=y+1$. Then
    $$P_{d,p}(y+1)=y^{2d}-p(y+1)\sum_{\ell=0}^{d-1}b_{d,\ell}y^{2\ell}(y+2)^{2d-2-2\ell}.$$
    All non-leading coefficients are divisible by $p$. The constant term is obtained only from $\ell=0$, and equals
    $$-p\,b_{d,0}\,2^{2d-2}=-p(d-1)!2^{2d-2}.$$
    Since $p>d-1$ and $p$ is odd, this constant term is not divisible by $p^2$. Hence $P_{d,p}(y+1)$ is Eisenstein at $p$. Therefore $P_{d,p}$ is irreducible over $\Z$.
\end{proof}
\subsection{Root location}
Since $P_{d,p}$ is reciprocal, define $R_{d,p}(y)$ by
$$x^{-d}P_{d,p}(x)=R_{d,p}(x+x^{-1}).$$
Using
$$x^{-1}(x-1)^2=x+x^{-1}-2,\qquad x^{-1}(x+1)^2=x+x^{-1}+2,$$
one obtains
\begin{equation}\label{eq:Rdp}
    R_{d,p}(y)=(y-2)^d-p\sum_{\ell=0}^{d-1} b_{d,\ell}(y-2)^\ell(y+2)^{d-1-\ell}.
\end{equation}
\begin{lemma}\label{lem:rootdist}
    For every $d\geq2$, there exists $p_0(d)$ such that for every prime $p\geq p_0(d)$, the polynomial $R_{d,p}(y)$ has one root in $(2,\infty)$ and $d-1$ roots in $(-2,2)$.
\end{lemma}
\begin{proof}
    Set
    $$z=\frac{y-2}{y+2},\qquad y=2\frac{1+z}{1-z}.$$
    Then $y>2$ corresponds to $0<z<1$, and $-2<y<2$ corresponds to $z<0$. Multiplying $R_{d,p}(y)=0$ by the nonzero factor $(1-z)^d/4^{d-1}$, we obtain the equivalent equation
    \begin{equation}\label{eq:Fpd}
        F_{d,p}(z)=4z^d-p(1-z)B_d(z)=0.
    \end{equation}
    First, $F_{d,p}(0)=-pB_d(0)<0$, while $F_{d,p}(1)=4>0$. Thus $F_{d,p}$ has at least one root in $(0,1)$ for every $p>0$.

    It remains to locate the negative roots. For $p\to\infty$, the polynomials
    $$p^{-1}F_{d,p}(z)=\frac4p z^d-(1-z)B_d(z)$$
    converge uniformly on compact subsets of $\C$ to 
    $$G_d(z)=-(1-z)B_d(z)=-(1-z)\prod_{j=1}^{d-1}(z+j).$$
    The negative roots $-1,-2,\ldots,-(d-1)$ are simple. Choose pairwise disjoint intervals
    $$I_j=(-j-\eta,-j+\eta)\subset(-\infty,0),\qquad j=1,\ldots,d-1,$$
    with $0<\eta<1/3$. Since each of these roots is simple, $G_d$ has opposite signs at the two endpoints of $I_j$. For all sufficiently large $p$, the same sign changes hold for $p^{-1}F_{d,p}$, hence for $F_{d,p}$. Therefore $F_{d,p}$ has a root in each $I_j$.

    We have found one root in $(0,1)$ and $d-1$ roots in $(-\infty,0)$. Since $F_{d,p}$ has degree $d$, these are all of its roots. Translating back to $y=2(1+z)/(1-z)$, the root in $(0,1)$ gives a root in $(2,\infty)$ and the negative roots give roots in $(-2,2)$.
\end{proof}
\begin{proof}[Proof of Theorem~\ref{thm:main1}]
    Choose an odd prime $p>d-1$ large enough for Lemma \ref{lem:rootdist}. By Lemma \ref{lem:realization}, the polynomial $P_{d,p}$ is the full reciprocal-radius growth polynomial of a connected $K_{2d+1}$-free graph. By Lemma \ref{lem:eisenstein}, it is irreducible, monic, reciprocal, and has constant term 1. By Lemma \ref{lem:rootdist} and Proposition \ref{prop:SalemCriterion}, it is a Salem polynomial of degree $2d$. There are infinitely many primes above any fixed bound, so infinitely many examples are obtained.
\end{proof}

\begin{example}[Degree four]
    For $d=2$, $B_2(z)=z+1$. Taking $p=3$ gives
    $P_{2,3}(x)=x^4-10x^3+6x^2-10x+1.$
    The associated polynomial is
    $R_{2,3}(y)=y^2-10y+4,$
    whose roots are approximately $0.4174$ and $9.5826$. Thus $P_{2,3}$ is a Salem polynomial.
\end{example}

\begin{example}[Degree six]
    For $d=3$, $B_3(z)=(z+1)(z+2)=z^2+3z+2$. The prime $p=5$ is not large enough for the root-distribution conclusion: it gives
    $$P_{3,5}(x)=x^6-36x^5-5x^4-80x^3-5x^2-36x+1$$
    and
    $$R_{3,5}(y)=y^3-36y^2-8y-8,$$
    whose one real root is approximately $36.2269$ and the other roots are non-real. Taking instead $p=19$ gives
    $$P_{3,19}(x)=x^6-120x^5-61x^4-248x^3-61x^2-120x+1$$
    and
    $$R_{3,19}(y)=y^3-120y^2-64y-8.$$
    The roots of $R_{3,19}$ are approximately $-0.3310$, $-0.2005$, and $120.5315,$ so $P_{3,19}$ is a Salem polynomial.
\end{example}

\section{Independence-polynomial obstructions}\label{sec:obstruction}
The obstruction in this section translates the arithmetic realization problem into the enumerative theory of independent sets. Recall that the independence polynomial of a graph $H$ is
$$I_H(u)=\sum_{k=0}^{\alpha(H)}i_ku^k,$$
where $i_k$ is the number of independent sets of size $k$. This is a standard graph polynomial; see the survey of Levit and Mandrescu \cite{LevitMandrescu}. Since cliques in $\Gamma$ are independent sets in the complement graph $\overline\Gamma$, the clique polynomial $C_\Gamma(u)$ is exactly $I_{\overline\Gamma}(u)$.

\subsection{The inverse transform}
For a monic polynomial $M(x)$ of degree $D$, define
$$A_M(u)=(-u)^D M\!\left(-1-\frac1u\right).$$
This is an integral polynomial whenever $M\in\Z[x]$. The inverse relation is
\begin{equation}\label{eq:inverseA}
    M(x)=(x+1)^D A_M\!\left(-\frac1{x+1}\right).
\end{equation}
If $M=P_\Gamma$, then by \eqref{eq:Pclique},
\begin{equation}\label{eq:AMequalsC}
   A_M(u)=C_\Gamma(u).
\end{equation}
Equivalently, since cliques in $\Gamma$ are independent sets in the complement graph $\overline\Gamma$,
\begin{equation}\label{eq:AMequalsI}
   A_M(u)=I_{\overline\Gamma}(u).
\end{equation}
Thus every strongly primitive right-angled Coxeter full reciprocal-radius growth polynomial must pass the test that $A_M$ is a graph independence polynomial.
\subsection{Vertex cover defect}
Let
$$M(x)=x^D+a_1x^{D-1}+\cdots+a_1x+1$$
be reciprocal. Expanding $A_M$ at the linear term gives
\begin{equation}\label{eq:linearCoeffA}
    [u]A_M(u)=D-a_1.
\end{equation}
If $A_M=I_H$ for a graph $H$, then $[u]A_M=|V(H)|$. Since $\deg(A_M)=D$, the independence number of $H$ is $D$. Hence the vertex cover number of $H$ is
\begin{equation}\label{eq:delta}
    \tau(H)=|V(H)|-\alpha(H)=(D-a_1)-D=-a_1.
\end{equation}
We call $\delta=-a_1$ the defect.

The number of maximum independent sets of a graph with vertex cover number $\delta$ is at most $2^\delta$. Indeed, if $S$ is a vertex cover with $|S|=\delta$, then $V(H)\setminus S$ is independent, and every maximum independent set is determined by its intersection with $S$. Therefore there are at most $2^\delta$ such sets.

Since the coefficient of $u^D$ in $A_M$ is $(-1)^DM(-1)$, we obtain the following obstruction.

\begin{proposition}\label{prop:elementary-obstruction}
    Let
    $$M(x)=x^D+a_1x^{D-1}+\cdots+a_1x+1$$
    be a Salem minimal polynomial. If $M$ is the full reciprocal-radius polynomial $P_\Gamma$ for some graph $\Gamma$, then
    \begin{equation}\label{eq:a1condition}
        a_1\leq-2
    \end{equation}
    and
    \begin{equation}\label{eq:topbound}
        0<(-1)^DM(-1)\leq 2^{-a_1}.
    \end{equation}
\end{proposition}
\begin{proof}
    If $M=P_\Gamma$, then $A_M$ is a graph independence polynomial. Hence the defect $-a_1$ is a vertex cover number and is nonnegative, so $a_1\leq0$.

    If $a_1=0$, then the complement graph has vertex cover number zero, so it is edgeless. Its independence polynomial is $(1+u)^D$. Using \eqref{eq:inverseA}, transforming back gives $M(x)=x^D$, not a Salem polynomial.

    If $a_1=-1$, then the complement graph has vertex cover number one. Such a graph is a star, possibly with isolated vertices. Hence its independence polynomial has the form
    $$(1+u)^s\left((1+u)^k+u\right)$$
    for some $s\geq0$ and $k\geq1$. Transforming back gives
    $$x^s\left(x^k-(x+1)^{k-1}\right).$$
    This polynomial is not reciprocal of even degree with constant term 1, and hence cannot be a Salem minimal polynomial. Therefore $a_1\leq-2$.

    Finally, $(-1)^DM(-1)$ is the leading coefficient of $A_M$, hence the number of maximum independent sets. It is positive and at most $2^{-a_1}$, as explained above.
\end{proof}

\section{A sharp first-coefficient obstruction}\label{sec:firstcoefficient}

Proposition \ref{prop:elementary-obstruction} uses only the linear and top coefficients of $A_M$. To prove the sharper bound needed for Theorem \ref{thm:main2}, we need some lemma.

Let
$$M(x)=x^D+a_1x^{D-1}+\cdots+a_1x+1$$
be monic and reciprocal, assume $M=P_\Gamma$, and put $H=\overline{\Gamma}$ and $\delta=\tau(H)=-a_1$. By the inverse transform, $I_H(u)=I_H(-1-u)$. Choose a minimum vertex cover $S$ of $H$ and put $T=V(H)\setminus S$. Thus, as in $\eqref{eq:delta}$, $|S|=\delta$ and $|T|=D$. for $K=\operatorname{Ind}(H[S])$ and $A\in K$, let $q(A)$ be the number of vertices of $T$ having no neighbor in $A$. Then 
\[
I_H(u)=\sum_{A\in K}u^{|A|}(1+u)^{q(A)}.
\]
The sets
\[
  Z=\{A\in K:q(A)=0\},\qquad
  U=\{A\in K:q(A)\le1\}
\]
are upper sets in $K$.  By the symmetry of $I_H$,
\[
 I_H(-1)=1,\qquad I_H'(-1)=-(D+\delta).
\]
Consequently,
\begin{align}
    \sum_{A\in Z}(-1)^{|A|}&=1,
    \label{eq:upper-alternating}\\
    I_H'(-1)&=\sum_{A\in Z}|A|(-1)^{|A|-1}+\sum_{A\in U\setminus Z}(-1)^{|A|}.
    \label{eq:upper-derivative}
\end{align}

\begin{lemma}\label{lem:small-cover}
If $2\le\delta\le4$, then $I_H'(-1)\ge-2\delta$. If $\delta=4$ and equality holds, then
\[
    K=2^S,\qquad Z=\{S\},\qquad U=\{A\subseteq S:|A|\ge3\}.
\]
\end{lemma}

\begin{proof}
Let $k_i$ and $z_i$ count the $i$-element faces of $K$ and $Z$.  Since $\varnothing\notin Z$, \eqref{eq:upper-alternating} gives $-z_1+z_2-z_3+z_4=1$. Only faces of odd cardinality contribute negatively in the second sum of \eqref{eq:upper-derivative}; hence
\begin{align}
    I_H'(-1) &\geq z_1-2z_2+3z_3-4z_4-(\delta-z_1)-(k_3-z_3)\notag\\
    &=-\delta-2-k_3+2(z_3-z_4).\label{eq:key-derivative-bound}
\end{align}
For $\delta=2$, this is at least $-4$.  For $\delta=3$, $k_3\le1$ and $z_4=0$, giving at least $-6$.

Now let $\delta=4$.  If $K\ne2^S$, flagness gives $k_3\le2$ and $z_4=0$.  If $k_3\le1$, \eqref{eq:key-derivative-bound} gives $I_H'(-1)\ge-7$. If $k_3=2$, every face of $K$ is contained in one of its two three-element faces; since $Z$ is a nonempty upper set, $z_3\ge1$.  Thus $I_H'(-1)\ge-6$ in this case. In particular, equality at $-8$ is impossible unless $K=2^S$.

Suppose therefore that $K=2^S$.  Put $L=K\setminus Z$ and $J=K\setminus U$, both lower sets, and write $f_X(t)=\sum_{A\in X}t^{|A|}$. As $f_K(-1)=0$ and $f_Z(-1)=1$, we have $f_L(-1)=-1$. Also $S\notin L$; thus, writing $l_i$ for the number of $i$-element faces of $L$, we have $l_2=l_1+l_3-2$. Equation \eqref{eq:upper-derivative} now yields
\[
    I_H'(-1)=-f_L'(-1)-1-f_J(-1)=l_1-l_3-5-f_J(-1).
\]
Since $L$ and $J$ are lower sets on four vertices,
\[
    l_3\leq\binom{l_1}{3}\le l_1, \qquad
    f_J(-1)=j_0-j_1+j_2-j_3\leq1-j_1+\binom{j_1}{2}\le3.
\]
Therefore $I_H'(-1)\geq-8$.  Equality forces $l_1=l_3$ and $f_J(-1)=3$.  Since $f_L(-1)=-1$, we have $l_1>0$; and $\binom{l_1}{3}<l_1$ for $1\leq l_1\le3$. Hence $l_1=l_3=4$ and $L=2^S\setminus\{S\}$. Moreover, $f_J(-1)=3$ forces $j_1=4$, $j_2=6$, and $j_3=0$. Thus $J=\{A\subseteq S:|A|\le2\}$, proving the asserted structure of $Z$ and $U$.
\end{proof}

\begin{theorem}\label{thm:sharp-first-coefficient}
If $P_\Gamma(x)=x^D+a_1x^{D-1}+\cdots+a_1x+1$ is an irreducible Salem polynomial, then $a_1\leq-5$. The bound is sharp.
\end{theorem}
\begin{proof}
By Proposition \ref{prop:elementary-obstruction}, $\delta=-a_1\geq2$.  If $\delta\le4$, Lemma~\ref{lem:small-cover} and $I_H'(-1)=-(D+\delta)$ imply $D\le\delta$. Since a Salem polynomial has even degree $D\geq4$, necessarily $D=\delta=4$, and equality holds in the lemma. Thus $S$ is independent, $q(S)=0$, and $q(A)=1$ whenever $|A|=3$. For $t\in T$, put
\[
 C_t=\{s\in S: \{s,t\}\notin E(H)\}.
\]
Every three-element subset $A\subset S$ is contained in exactly one $C_t$, and no $C_t$ equals $S$. As $|S|=|T|=4$, the four sets $C_t$ are precisely the four three-element subsets of $S$. Since both $S$ and $T$ are independent, $H\cong4K_2$. Hence $P_\Gamma(x)=(x-1)^4$, a contradiction.  Thus $\delta\geq5$.

For sharpness, take $\Gamma$ to be $O_4=K_{2,2,2,2}$ with one pendant vertex attached to an existing vertex. Then $C_\Gamma(u)=(1+2u)^4+u(1+u)$, and
\[
\begin{aligned}
    P_\Gamma(x)&=(x-1)^4-x(x+1)^2\\
    &=x^4-5x^3+4x^2-5x+1.
\end{aligned}
\]
In modulo $3$, this polynomial becomes $\Phi_5(x)=x^4+x^3+x^2+x+1$, which is irreducible because $3$ has multiplicative order $4$ modulo $5$. Finally, with $y=x+x^{-1}$,
\[
    x^{-2}P_\Gamma(x)=y^2-5y+2,
\]
whose roots $(5\pm\sqrt{17})/2$ lie in $(-2,2)$ and $(2,\infty)$, respectively. By Proposition~\ref{prop:SalemCriterion}, $P_\Gamma$ is a Salem polynomial, attaining $a_1=-5$.
\end{proof}

\section{A growth gap and Lehmer's number}\label{sec:Lehmer}

The preceding sections concern realization of a prescribed Salem polynomial as the full reciprocal-radius growth polynomial. For growth rates themselves, an elementary gap gives a stronger exclusion for small Salem numbers. We write $\operatorname{gr}(W_\Gamma,S)$ for the standard spherical growth rate of $W_\Gamma$, where $S=V(\Gamma)$.

\begin{proposition}\label{prop:golden-ratio-gap}
If a right-angled Coxeter group $W_\Gamma$ has exponential growth, then
\[
    \operatorname{gr}(W_\Gamma,S)\geq\varphi=\frac{1+\sqrt{5}}{2}.
\]
The bound is sharp.
\end{proposition}

\begin{proof}
Put $H=\overline{\Gamma}$. If every vertex of $H$ has degree at most one, then $H$ is a disjoint union of edges and isolated vertices. Hence $\Gamma$ is a join of copies of $O_1$ and $K_1$, so that
\[
    W_\Gamma\cong D_\infty^m\times(\mathbb Z/2\mathbb Z)^s
\]
for some $m,s\geq0$, where $D_\infty$ is the infinite dihedral group. Such a group has polynomial growth.

Thus exponential growth implies that $H$ has a vertex $v$ with two distinct neighbors $a,b$. The induced subgraph $\Lambda=\Gamma[\{a,b,v\}]$ is either $K_2\sqcup K_1$ or $3K_1$.
By Proposition~\ref{prop:SteinbergRA}, their spherical growth series are respectively
\[
    \frac{(1+t)^2}{1-t-t^2} \qquad\text{and}\qquad\frac{1+t}{1-2t},
\]
with growth rates $\varphi$ and $2$. The special subgroup $W_\Lambda\leq W_\Gamma$ is length-preserving with respect to the standard generating sets. Therefore $\operatorname{gr}(W_\Gamma,S)\geq\operatorname{gr}(W_\Lambda,V(\Lambda))\geq\varphi$. Equality is attained when $\Gamma=K_2\sqcup K_1$.
\end{proof}

Consequently, no Salem number in $(1,\varphi)$ can be a standard spherical growth rate of a right-angled Coxeter group. In particular, Lehmer's number $\lambda_0=1.1762808\ldots<\varphi$ is excluded. It rules out $\lambda_0$ as a growth rate even if its minimal polynomial is allowed to be a proper factor of $P_\Gamma$.

\noindent\textbf{AI disclosure.}
OpenAI's ChatGPT was used to assist with exposition and the refinement of mathematical arguments. The author takes full responsibility for the final content.

\end{document}